\documentclass[12pt]{amsart}
\usepackage{amsthm}
\usepackage{amssymb}
\usepackage{amsmath}
\theoremstyle{plain}
\newtheorem{proposition}{Proposition}
\newtheorem{theorem}[proposition]{Theorem}

\newtheorem{corollary}[proposition]{Corollary}

\theoremstyle{definition}

\theoremstyle{definition}

\newtheorem{remark}[proposition]{Remark}

\numberwithin{equation}{section}

\numberwithin{proposition}{section}

\gdef\myletter{}

\let\savetheequation\theequation
\def\theequation{\savetheequation\myletter}

\newcommand{\CC}{{\mathbb C}}

\newcommand{\RR}{{\mathbb R}}

\newcommand{\PP}{{\mathbb P}}
\newcommand{\BB}{{\mathbb B}}

\renewcommand{\Im}{\mbox{Im}}

\renewcommand{\Re}{\mbox{Re}}

\def \bar{\overline}

\newcounter{rev}
\newcounter{sq}
\begin{document}

\vskip 3mm

\title[Monge-Amp\`ere Measures for Convex Bodies]{\bf Monge-Amp\`ere Measures for Convex Bodies and Bernstein-Markov
Type Inequalities}

\author{D. Burns, N. Levenberg, S. Ma'u and Sz. R\'ev\'esz}{\thanks
{Supported in part by NSF grants DMS-0514070 (DB).\\ \indent
Supported in part by the Hungarian National Foundation for
Scientific Research, Project \#s T-049301 T-049693 and K-61908
(SzR). \\ \indent This work was accomplished during the fourth
author's stay in Paris under his Marie Curie fellowship, contract
\# MEIF-CT-2005-022927.}}

\address{Univ. of Michigan, Ann Arbor, MI 48109-1043 USA}
\email{dburns@umich.edu}

\address{Indiana University, Bloomington, IN 47405 USA}
\email{nlevenbe@indiana.edu}

\address{Indiana University, Bloomington, IN 47405 USA}
\email{sinmau@indiana.edu}

\address{ A. R\'enyi Institute of Mathematics, Hungarian Academy of Sciences,
\newline \indent Budapest, P.O.B. 127, 1364 Hungary.}
\email{revesz@renyi.hu}

\address{\hskip 1cm and
\newline \indent Institut Henri Poincar\'e, 11 rue Pierre et Marie Curie, 75005 \newline \indent Paris, France}
\email{revesz@ihp.jussieu.fr}

\date{\today}



\maketitle

\begin{abstract}
We use geometric methods to calculate a formula for the complex
Monge-Amp\`ere measure $(dd^cV_K)^n$, for $K \Subset \RR^n \subset
\CC^n$ a convex body and $V_K$ its Siciak-Zaharjuta extremal
function. Bedford and Taylor had computed this for symmetric
convex bodies $K$. We apply this to show that two methods for deriving Bernstein-Markov-type
inequalities, i.e., pointwise estimates of gradients of
polynomials, yield the same results for all convex bodies. A key role is played by
the geometric result that the extremal inscribed ellipses appearing
in approximation theory are the maximal area ellipses determining
the complex Monge-Amp\`ere solution $V_K$.
\end{abstract}

\section{Introduction.}\label{sec:intro}


For a function $u$ of class $C^2$ on a domain $\Omega\subset
\CC^n$, the complex Monge-Amp\`ere operator applied to $u$ is 
$$(dd^cu)^n:= i\partial \bar \partial u \wedge \cdots \wedge i\partial \bar \partial u.$$
For a plurisubharmonic (psh) function $u$ which is only locally
bounded, $(dd^cu)^n$ is well-defined as a positive measure. Given
a bounded set $E\subset \CC^n$, we define the Siciak-Zaharjuta
extremal function 
$$V_E(z):=\sup \{u(z):u\in L(\CC^n), \ u\leq 0 \ \hbox{on} \ E\}$$
where $L(\CC^n)$ denotes the class of psh functions $u$ on $\CC^n$
with $u(z)\leq \log ^+{|z|}+ c(u)$. If $E$ is non-pluripolar, the
upper-regularized function 
$$V_E^*(z):=\limsup_{\zeta\to z}V_E(\zeta)$$
is a locally bounded psh function which satisfies
$(dd^cV_E^*)^n=0$ outside of $\bar E$ and the total mass of
$(dd^cV_E^*)^n$ is $(2\pi)^n$.

In this paper we consider $E=K\subset \RR^n$ a convex body, that
is, a compact, convex set with non-empty interior. In this
situation the function $V_K = V^*_K$ is continuous but it is not
necessarily smooth, even if $K$ is smoothly bounded and strictly
convex. Indeed, for $K = \BB^n_{\RR}$, the unit ball in $\RR^n
\subset \CC^n$, Lundin found \cite{lunpre}, \cite{bar} that 
\begin{equation}
\label{eq:realball} V_K(z) =\frac{1}{2} \log h(|z|^2 + |z\cdot z -
1|),
\end{equation}
where $|z|^2 = \sum |z_j|^2, z \cdot z = \sum z_j^2,$ and
$h(\frac{1}{2}(t + \frac{1}{t})) = t,$ for $1 \leq t \in \RR.$ In
this example, the Monge-Amp\`ere measure $(dd^cV_K)^n$ has the
explicit form
$$(dd^cV_K)^n = n! \hbox{vol}(K)\frac{dx}{(1-|x|^2)^{\frac{1}{2}}}:=n! \hbox{vol}(K)\frac{dx_1 \wedge \cdots \wedge dx_n}{(1- |x|^2)^{\frac{1}{2}}} .$$

The main result of this paper is a general formula for this
measure (see Theorem \ref{main1} and Corollary \ref{cor:mt}).

\begin{theorem}
\label{main} Let $K$ be a convex body and $V_K$ its
Siciak-Zaharjuta extremal function. The limit
\begin{equation}
\label{dblim} \delta_B^K(x,y)=\delta_B (x,y):=\lim_{t\to
0^+}{V_K(x+ity)\over t}
\end{equation}
exists for each $x\in K^o$ and $y\in \RR^n$ and for $x\in
K^o$
\begin{equation}
    \label{eq:monge}
    (dd^cV_K)^n=\lambda(x)dx  \ \hbox{where} \ \lambda(x)=n!{\rm vol} (\{y: \delta_B(x,y)\leq 1\}^*).
    \end{equation}
    Moreover, $(dd^cV_K)^n$ puts no mass on the boundary $\partial K$ (relative to $\RR^n$).
\end{theorem}
\noindent Here, for a symmetric convex body $E$ in $\RR^n$,
$$E^*:=\{y\in \RR^n:x\cdot y\leq 1, \ \hbox{for all} \ x\in E\}$$
is also a symmetric convex body in $\RR^n$, called the polar of
$E$. The quantity $\delta_B(x,y)$ is continuous on $K^o \times
\RR^n$ and for each fixed $x\in K^o$, $y\to \delta_B(x,y)$ is a
norm on $\RR^n$; i.e., $\delta_B(x,y)\geq 0$; $\delta_B(x,\lambda
y)=\lambda \delta_B(x,y)$ for $\lambda \geq 0$; and
$\delta_B(x,y_1+y_2)\leq \delta_B(x,y_1)+ \delta_B(x,y_2)$ (see
\cite{bt}).

For a symmetric convex body, i.e., $K=-K$, Bedford and Taylor
\cite{bt} showed the existence of the limit (\ref{dblim}) and
proved the formula (\ref{eq:monge}) using the description of the
Monge-Amp\`ere solution given by Lundin \cite{lun}.  The present
paper relies on the description of $V_K$ given in \cite{blm},
\cite{blm2} for general convex bodies $K$. \cite{blm} showed the
existence, through each point $z \in \CC^n \setminus K$, of a
holomorphic curve on which $V_K$ is harmonic, while \cite{blm2}
showed that for many $K$ (all $K$ in $\RR^2$) these curves give a
continuous foliation of $\CC^n \setminus K$ by holomorphic curves.
It also showed that these curves are algebraic curves of degree 2,
and interpreted them in terms of a (finite dimensional)
variational problem among real ellipses contained in $K$.  The
real points of such a quadratic curve describe an ellipse within
$K$ of maximal area in its class of competitors. These competitor
classes are specified by the points $c$ on the hyperplane $H$
at infinity in $\PP^n$ through which the quadratic curves pass.
The geometry of these foliations is our main tool.

The norm $\delta_B(x,y)$ is also related to Bernstein-Markov
inequalities for real, multivariate polynomials on $K$. This will
be explained in section 3, specifically, in equation
(\ref{markov}) and the remarks after it. Conversely, a key role in
the proof of the main result Theorem \ref{main} is played by the
observation (Proposition \ref{maxarea}) that the extremal
inscribed ellipses appearing in a geometric approach -- the
``inscribed ellipse method" of Sarantopoulos \cite{sar},
\cite{revsar} --  to Bernstein-Markov inequalities are the maximal
area ellipses appearing in the determination of the Monge-Amp\`ere
solution as described above. A corollary of our main result is
that the inscribed ellipse method and the pluripotential-theoretic
method, due to Baran \cite{bar, bar2} for obtaining
Bernstein-Markov-type estimates are equivalent for all convex
bodies. This was straightforward for symmetric convex bodies. It
was proved for simplices and conjectured for the general case as ``Hypothesis A'' in
\cite{rev2}.

The remainder of the paper is organized as follows. In section
\ref{sec:rev} we recall in more detail some of the features of the
leaf structure for the Monge-Amp\`ere foliation. In section
\ref{sec:inscribed} we review the maximal inscribed ellipse
problem from \cite{sar}, \cite{revsar}, its relation to
Bernstein-Markov inequalities from approximation theory and to
the Monge-Amp\`ere maximal ellipses in section \ref{sec:rev}. We
also sketch the relation to the extremal function $V_K$ for
symmetric convex bodies \cite{bt}, \cite{bar2}. Finally, in
section \ref{sec:main}, using details of the Monge-Amp\`ere
foliation and its continuity, we prove the main results.

\section{Review of the variational problem.}
    \label{sec:rev}

\vskip 3mm

Let $K \subset \RR^n \subset \CC^n$ be a convex body, and consider
$\CC^n \subset \PP^n$, the complex projective space with $H :=
\PP^n \setminus \CC^n$ the hyperplane at infinity. Let $\sigma:
\PP^n \rightarrow \PP^n$ be the anti-holomorphic map of complex
conjugation, which preserves $\CC^n$ and $H$, and is the identity
on $\RR^n$. Let $H_{\RR}$ denote the real points of $H$ (fixed
points of $\sigma$ in $H$). For any non-zero vector $c \in \CC^n$,
let $\sigma(c) = \bar{c},$ and $[c] \in H$ the point in $H$ given
by the direction of $c$. If $[c] \neq [\bar{c}]$, then $c,
\bar{c}$ span a complex subspace $V \subset \CC^n$ of dimension
two, which is real, that is, invariant under $\sigma$; hence $V$
is the complexification of a two-dimensional real subspace $V_0
\subset \RR^n$. If we translate $V$ by a vector $A \in \RR^n$, we
get a complex affine plane $V + A$ invariant by $\sigma$ and
containing the real form $V_0 + A$, the fixed points of $\sigma$
in $V + A$. Associated to the point $[c] \in H$, we consider
holomorphic maps $f:
\triangle \rightarrow \PP^n$, $\triangle$ the unit disk in $\CC$,
such that $f(0) = [c]$, and $f(\partial \triangle) \subset K$.
Such maps can be extended by Schwarz reflection to maps (still
denoted) $f: \PP^1 \rightarrow \PP^n$ by the formula 
\begin{equation}
    \label{eq:reality}
    f(\tau(\zeta)) = \sigma(f(\zeta)) \in \PP^n
\end{equation}
where $\tau: \PP^1 \rightarrow \PP^1$ is the inversion $\tau(\zeta) =
1/\bar{\zeta}.$ In particular, such maps have the form 
\begin{equation}
    \label{eq:ansatz}
    f(\zeta) = \rho \frac{C}{\zeta} + A + \rho \bar C \zeta,
\end{equation}
where $[c] = [C],$ i.e.,
$C = \lambda c$, for some $\lambda \in \CC$, $A \in \RR^n$, and
$\rho
>0$. Then $f(\PP^1) \subset \PP^n$ is a quadratic curve, and
restricted to $\partial \triangle$, the unit circle in $\CC, f$
gives a parametrization of a real ellipse inside the planar convex
set $K \cap \{V_0 + A\}$, with center at $A$. According to
\cite{blm2}, the extremal function $V_K$ is harmonic on the
holomorphic curve $f(\triangle \setminus \{0\}) \subset \CC^n
\setminus K$ if and only if the area of the ellipse bounded by
$f(\partial \triangle)$ is {\em maximal} among all those of the
form (\ref{eq:ansatz}).

For a fixed, normalized $C$, this is equivalent to varying $A\in
\RR^n$ and $\rho >0$ among the maps in (\ref{eq:ansatz}) with
${\mathcal E}=f(\partial \triangle) \subset K$ in order to
maximize $\rho$. Fixing $C$ amounts to prescribing the orientation
(major and minor axis) and eccentricity of a family of inscribed
ellipses in $K$. We will call an extremal ellipse ${\mathcal E}$ a
{\sl maximal area ellipse}, or simply {\sl $a-$maximal}. In the
case where $\partial K$ contains no parallel faces, for each
$[c]\in H$ there is a unique $a-$maximal ellipse (Theorem 7.1,
\cite{blm2}); we denote the corresponding map by $f_c$. In this
situation, the collection of complex ellipses $\{f_c(\triangle
\setminus \{0\}) : [c] \in H\}$ form a continuous foliation of
$\CC^n \setminus K$. In simple terms, this means that if $z,z'$
are distinct points in $\CC^n \setminus K$, with $|z-z'|$ small,
lying on leaves $L(z):=f_c(\triangle \setminus \{0\})$ and
$L'(z):=f_{c'}(\triangle \setminus  \{0\})$, then the
corresponding leaf parameters in (\ref{eq:ansatz}) are close;
i.e., $C\sim C'$, $\rho \sim \rho'$ and $A\sim A'$ (and of course
$|\zeta_z| \sim |\zeta_{z'}|$ where $f_c(\zeta_z)=z$ and
$f_{c'}(\zeta_{z'})=z'$). Any convex body in $\RR^2$ admits a
continuous foliation; this follows from Proposition 9.2 or Theorem
10.2 in \cite{blm2}. Moreover, if we let ${\mathcal C}$ denote the
set of all convex bodies $K\subset \RR^n$ admitting a continuous
foliation, then ${\mathcal C}$ is dense in the Hausdorff metric in
the set ${\mathcal K}$ of all convex bodies $K\subset \RR^n$. This
follows, for example, from the fact that strictly convex bodies
$K$ belong to ${\mathcal C}$ (cf., Theorem 7.1 of  \cite{blm2}).
In addition, all symmetric convex bodies admit a continuous
foliation.

For convenience, instead of using the holomorphic curves
$f(\triangle \setminus \{0\})$ we will work with the holomorphic
curve $f(\CC \setminus \bar \triangle)$; thus $V_K$ being harmonic
on this curve means that
\begin{equation}\label{harmonic}
V_K(f(\zeta)) =\log |\zeta| \ \hbox{for} \ |\zeta|\geq 1.
\end{equation}

\section{Inscribed ellipse problem.}
    \label{sec:inscribed}

\vskip 3mm

Let $K\subset \RR^n$ be a convex body. Consider the following
geometric problem: fix $x\in K^o$ and a non-zero vector $y\in
\RR^n$ and consider all ellipses ${\mathcal E}$ lying in $K$ which
contain $x$ and have a tangent at $x$ in the direction $y$. We
write $y\in T_x{\mathcal E}$. That is, ${\mathcal E}={\mathcal
E}_b={\mathcal E}_b(x,y)$ is given by a parameterization
\begin{equation}
\label{theta} \theta \to r(\theta):=a\cos \theta + by\sin \theta
+(x-a)
\end{equation}
where $a\in \RR^n$ and $b\in \RR^+$ are such that $r(\theta)\in K$
for all $\theta$. {\it The problem is to maximize $b$ among all
such ellipses}. We will call such an ellipse a {\sl maximal
inscribed ellipse} (for $x,y$) or simply {\sl $b-$maximal}. Note
that $r(0)=x$ and $r'(0)=by$; thus one is allowed to vary $a$ and
$b$ in (\ref{theta}). Often we will normalize and assume that $y$
is a unit vector. An observation which will be used later is that
if we fix $a$, then ${\mathcal E}_b$ lies ``inside'' ${\mathcal
E}_{b'}$ if $b < b'$ with two common points $x$ and $x-2a$.

\vskip4pt

We give some motivation for studying this problem; this goes back
to Sarantapolous (cf., \cite{sar} or \cite{revsar}). For any such
ellipse ${\mathcal E}$, if $p$ is a polynomial of $n$ real
variables of degree $d$, say, with $||p||_K\leq 1$, then
$t(\theta):=p(r(\theta))$ is a trigonometric polynomial of degree
at most $d$ with $||t||_{[0,2\pi]}\leq ||p||_K \leq 1$ (since
${\mathcal E}\subset K$). By the Bernstein-Szeg\"o inequality for
trigonometric polynomials,
$${|t'(\theta)|\over \sqrt {||t||_{[0,2\pi]}^2-t(\theta)^2} }\leq d.$$
From the chain rule,
$$|t'(0)|=|\nabla p (x)\cdot r'(0)|= |D_{by}p(x)|=b|D_yp(x)|.$$
Thus
$$b|D_yp(x)|=|t'(0)|\leq  d\sqrt {||t||_{[0,2\pi]}^2-t(0)^2}\leq d\sqrt {1-p(x)^2};$$
i.e.,
\begin{equation}
\label{markov} {1\over \hbox{deg} \; p}{|D_yp(x)|\over \sqrt
{1-p(x)^2}}\leq {1\over b}.
\end{equation}
The left-hand-side is related to what we shall refer to as a {\it
Bernstein-Markov inequality}\footnote{For any univariate algebraic
polynomial of degree not exceeding $n$, the sharp uniform estimate
for the derivative $\|p'\|_{\infty,[-1,1]}\leq
n \|p\|_{\infty,[-1,1]}$ is due to Markov, 
while the pointwise estimate $|p'(x)|\sqrt{1-x^2} \leq
n \|p\|_{\infty,[-1,1]}$ is known as Bernstein's Inequality, see
e.g. \cite{borer}, pages 232-233.
In approximation theory, these types of derivative estimates -- or, in the
multivariate case, gradient and directional derivative
estimates -- are usually termed Bernstein and/or Markov type
inequalities.}: it relates the directional derivative of $p$ at
$x$ in the direction $y$ with the sup-norm of $p$ on $K$ (the
``1'' on the right-hand-side of (\ref{markov})) and the degree of
$p$. This motivates the definition of the {\it Bernstein-Markov
pseudometric} (cf., \cite{blw}): given $x\in K$, $y\in \RR^n$, let
$$\delta_M^K(x;y)=\delta_M(x;y):=\sup_{||p||_K\leq 1, \ {\rm deg} \; p \geq 1}{1\over {\rm deg} \; p}{|D_yp(x)|\over \sqrt {1-p(x)^2}}.$$
(This definition makes sense for general compacta in $\RR^n$).
Inequality (\ref{markov}) says that {\sl whenever you have an
inscribed  ellipse ${\mathcal E}_b={\mathcal E}_b(x,y)$ through
$x$ with tangent at $x$ in the direction of $y$, the number $1/b$
gives an upper bound on the Bernstein-Markov pseudometric}:
$$\delta_M(x;y) \leq 1/b.$$
The bigger you can make $b$, the better estimate you have.

\vskip4pt

Let
\begin{equation}
\label{defofb} b^*(x,y):=\sup \{b: {\mathcal E}_b(x,y) \subset
K\}.
\end{equation}
Note that $b^*(x,ty)=b^*(x,y)/t$ for $t>0$. In the symmetric case,
this is intimately related to $V_K$:

\begin{proposition}
\label{symmcase} If $K=-K$, then $\delta_M(x;y) ={1\over
b^*(x,y)}$. Moreover,
$$ \delta_M(x;y)  = \lim_{t\to 0^+} {V_K(x+ity)\over t}.$$
\end{proposition}

As in the introduction, define
\begin{equation}
\label{delta} \delta_B^K(x;y)=\delta_B(x,y):=\lim_{t\to 0^+}
{V_K(x+ity)\over t},
\end{equation}
provided this limit exists. For symmetric convex bodies $K$, the
proposition says that the limit does exist and we have
\begin{equation}
\label{db=dm} \delta_B(x,y) = \delta_M(x,y) ={1\over b^*(x,y)}.
\end{equation}
Moreover, for each fixed $x\in K^o$, the function $y\to {1\over
b^*(x,y)}$ is a norm (cf.,
Proposition \ref{dbnorm}). A proof of the existence of the  limit
was given by Bedford and Taylor \cite{bt}. We sketch an alternate
proof due to Baran \cite{bar2}.

\vskip4pt {\sl Step 1}: $V_K(z)=\sup \{\log |h(z\cdot Z))|: Z\in
K^*\}$ where $h(w)=w+\sqrt {w^2-1}$ is the standard Joukowski map
and
$$K^*:=\{Z:x\cdot Z\leq 1 \ \hbox{for all} \ x\in K\}$$ is the polar of $K$ (cf., \cite{lun}, or \cite{bar2}, Proposition 1.15). \vskip4pt

{\sl Step 2}: We have the following explicit estimates on $h$: if
$|\alpha|<1$, $|\beta|\leq \sqrt {1-|\alpha|}$, and $0<\epsilon
\leq 1/2$, then
$$(1-\epsilon){|\beta|\over \sqrt {1-\alpha^2}}\leq {1\over \epsilon}\log |h(\alpha +i\epsilon \beta)| \leq {|\beta|\over \sqrt {1-\alpha^2}}$$
(the inequality on the right-hand-side is valid without the
restriction $|\beta|\leq \sqrt {1-|\alpha|}$; cf., \cite{bar2}, Proposition 1.13). This states precisely
that $\log |h|$ is Lipschitz as you approach $(-1,1)$ vertically
and the Lipschitz constant grows like one-over-the-distance to the
boundary points. \vskip4pt

Now fix $x\in K^o$ and $y\in \RR^n$; then for any $Z\in K^*$ and
for $t>0$ small, since $(x+ity)\cdot Z= x\cdot Z+ity\cdot Z$,
\begin{equation}
\label{barest} (1-\epsilon){1\over t}{t|y\cdot Z|\over \sqrt
{1-(x\cdot Z)^2}}\leq {1\over t}\log |h((x+ity)\cdot Z)| \leq
{1\over t}{t|y\cdot Z|\over \sqrt {1-(x\cdot Z)^2}}.
\end{equation}
This gives
\begin{equation}
\label{vklim} \lim_{t\to 0^+} {V_K(x+ity)\over t}=\sup \{{|y\cdot
Z|\over \sqrt {1-(x\cdot Z)^2}}: Z\in K^*\}.
\end{equation}

\vskip4pt To relate this with $b^*(x;y)$, in the symmetric case,
the $b-$maximal ellipse is easily seen to be an $a-$maximal
ellipse (see the next proposition for a generalization of this),
and the linear polynomial $p$ that maps the support ``strip'' of
this ellipse to $[-1,1]$ (i.e., it maps one parallel support
hyperplane to $-1$ and the other to $+1$) is easily seen to give
$${1\over {\rm deg} \; p}{|D_yp(x)|\over \sqrt {1-p(x)^2}}={1\over b^*(x,y)}$$
so that we have the equality $\delta_M(x,y)={1\over b^*(x,y)}$.
Thus, the first part of the proposition is proved. Moreover, we
have the following formula for $b^*(x,y)$:
$$b^*(x;y)=\inf \{ {\sqrt {1-(x\cdot w)^2}\over |y\cdot w|}: w\in K^*\}. $$
To see this, in the symmetric case one considers symmetric
ellipses in (\ref{theta}), i.e., $a:=x$, and, from the definition
of $b^*(x;y)$ and $K^*$ we can write
$$b^*(x;y)=\sup \{b: \sup_{w\in K^*, \ t\in [0,2\pi]} |x\cos t\cdot w +yb\sin t\cdot w|=1\}$$
$$=\sup \{b: \sup_{w\in K^*} [(w\cdot x)^2+b^2(w\cdot y)^2]= 1\}.$$
Basically unwinding things shows that this is the reciprocal of
(\ref{vklim}). For details we refer the reader to \cite{rev} or
\cite{blw}.

\vskip4pt

A  key geometric observation which will be used in the next
section is  the following.

\begin{proposition}
\label{maxarea} For any convex body $K$, a $b-$maximal ellipse
${\mathcal E}$ is also an $a-$maximal ellipse.

\end{proposition}

\begin{proof} First observe that an $a-$maximal ellipse ${\mathcal E}$ is characterized by  the property that {\it no translate of ${\mathcal E}$ lies entirely in the interior $K^o$ of $K$}. For if ${\mathcal E}+v\subset K^o$ for some $v\not =0$, one can dilate ${\mathcal E}+v$ to get an ellipse with the same orientation and eccentricity as ${\mathcal E}$ which lies in $K$ but has larger area. Conversely, if ${\mathcal E}$ is not an $a-$maximal ellipse, then one can find an ellipse ${\mathcal E}'$ with the same orientation and eccentricity as ${\mathcal E}$ which lies in $K$ but has larger area. The convex hull $H$ of ${\mathcal E}\cup{\mathcal E}'$ lies in $K$ and we can translate ${\mathcal E}$ within $H$ to an ellipse ${\mathcal E}''$ lying in the two-dimensional surface $S({\mathcal E}')$ determined by ${\mathcal E}'$; if ${\mathcal E}''$ does not lie in the ``interior'' of $S({\mathcal E}')$, we simply translate it within this surface (since the area of ${\mathcal E}'$ !
 is greater than that of ${\mathcal E}''$) until it does.

Indeed, we need a slightly more precise statement: ${\mathcal E}$
is not an $a-$maximal ellipse if and only if there is a unit
vector $v$ and $\delta >0$ such that ${\mathcal E}+sv\subset K^o$
for $0<s < \delta$, i.e., all translates by a small amount in some
direction stay in $K^o$. This follows from the previous paragraph
if we observe the following fact: if $K$ is a convex body,
$u\in K$ and $u+v\in \partial K^o$, then the entire half-open
segment $(u,u+v]$ lies in $K^o$.

Suppose that ${\mathcal E}$ given by
$$\theta \to a\cos \theta + by\sin \theta +(x-a)$$
is a $b-$maximal ellipse for $x,y$. For the sake of obtaining a
contradiction, we assume that ${\mathcal E}$ is not an $a-$maximal
ellipse. By the previous paragraph, we can find a nonzero vector
$v$ and $\delta>0$ so that ${\mathcal E}_s:={\mathcal E}+sv$ lies
in $K^o$ for $0<s < \delta$. For $0<\epsilon <\delta/2$, consider
the ellipse $\tilde {\mathcal E}(\epsilon)$ given by
\begin{equation}
\label{repsilon} r_{\epsilon}(\theta)= (a-\epsilon v)\cos \theta +
by\sin \theta +x-(a-\epsilon v).
\end{equation}
We claim that $\tilde {\mathcal E}(\epsilon) \subset K^o$.
Assuming this is the case, note that $r_{\epsilon}(0)=x\in \tilde
{\mathcal E}(\epsilon)$ and $r_{\epsilon}'(0)=by$; in particular,
the ``$b$'' for $\tilde {\mathcal E}(\epsilon)$ is the same as the
``$b$'' for ${\mathcal {\mathcal E}}$. Since $\tilde {\mathcal
E}(\epsilon) \subset K^o$, we can modify $\tilde {\mathcal
E}(\epsilon)$ to an ellipse $\tilde {\mathcal E}(\epsilon)'$
containing $x$ and lying in $K$ by replacing $b$ in
(\ref{repsilon}) by $b'>b$ contradicting the assumption that
${\mathcal E}$ is a $b-$maximal ellipse for $x,y$.

\vskip4pt To verify that $\tilde {\mathcal E}(\epsilon) \subset
K^o$, observe that for each fixed $\theta$, the point
$$(a-\epsilon v)\cos \theta + by\sin \theta +x-(a-\epsilon v)$$
$$=a\cos \theta + by\sin \theta +(x-a)+\epsilon v(1-\cos \theta)$$
on $\tilde {\mathcal E}(\epsilon)$ lies on the ellipse ${\mathcal
E}_{s_{\theta}}:={\mathcal E} +\epsilon (1-\cos \theta)v$ where
$s_{\theta}=\epsilon (1-\cos \theta)\leq 2\epsilon <\delta$. Thus
$\tilde {\mathcal E}(\epsilon) \subset K^o$.

\end{proof}

For use in the next section, we prove some results about the
function $b^*(x,y)$.

\begin{proposition}
\label{dbnorm} For a convex body $K\subset \RR^n$, $b^*(x,y)$
defined in (\ref{defofb}) is a continuous function of $x\in K^o$
and $y\in \RR^n$. Moreover, for each fixed $x\in K^o$, $y\to
1/b^*(x,y)$ is a norm in $\RR^n$.
\end{proposition}

\begin{proof} For the continuity of $b^*(x,y)$, we first verify uppersemicontinuity of this function. Fix a convex body $K$ and fix $x\in K^o$ and $y\in \RR^n$. Let $\{x_j\}\subset K^o$ with $x_j \to x$ and $\{y_j\}\subset \RR^n$ with $y_j \to y$. Let
$$r_j(\theta)= a_j \cos \theta +b^*(x_j,y_j)y_j\sin \theta+(x_j-a_j)$$
parameterize a $b-$maximal ellipse ${\mathcal E}_j$ for $K$
through $x_j$ in the direction $y_j$. Take a subsequence $\{j_k\}$
of positive integers so that the numbers
$\{b^*(x_{j_k},y_{j_k})\}$ converge to a number $\tilde b$; and
take a further subsequence (which we still call $\{j_k\}$) so that
the vectors $\{a_{j_k}\}\subset \RR^n$ converge to $a\in \RR^n$.
Consider the ellipse ${\mathcal E}$ where
$$r(\theta)= a\cos \theta+\tilde by\sin \theta+(x-a).$$
Since $x_{j_k}\to x, \ y_{j_k}\to y, \ b^*(x_{j_k},y_{j_k}) \to
\tilde b$ and $a_{j_k} \to a$, the functions $r_{j_k}$ converge
uniformly to $r$ (equivalently, the ellipses ${\mathcal E}_{j_k}$
converge in the Hausdorff metric to ${\mathcal E}$). Thus
${\mathcal E}$ is an inscribed ellipse for $K$ through $x$ in the
direction of $y$; hence $\tilde b \leq b^*(x,y)$; i.e.,
$$\limsup_{x'\to x, \ y'\to y} b^*(x',y')\leq b^*(x,y).$$

To verify lowersemicontinuity of $b^*(x,y)$, we fix $x\in K^o, \
y\in \RR^n$ and $b'< b^*(x,y)$, and we show there is a $\delta >0$
such that for all  $|x'-x| < \delta, \ |y'-y| < \delta$ there is
an inscribed ellipse ${\mathcal E}'$ through $x'$ with tangent
direction $y'$ of the form
$$\theta \to a'\cos \theta+b'y'\sin \theta+(x'-a').$$

Let ${\mathcal E}$ be a $b-$maximal ellipse through $x$ in the
direction $y$ given by
$$r(\theta)= a\cos \theta+b^*(x,y)y\sin \theta+(x-a).$$
If $x-2a\in K^o$, then for $b'<b$ the ellipse ${\mathcal E}_{b'}$
$$r_{b'}(\theta)= a\cos \theta+b'y\sin \theta+(x-a)$$
lies fully in $K^o$ (for ${\mathcal E}_{b'}$ lies entirely
``inside'' of ${\mathcal E}$ except for the common points
$x,x-2a$, and $x\in K^o$). Then any sufficiently small
translation ${\mathcal E}'$
$$\theta \to a\cos \theta+b'y\sin \theta+(x'-a)$$
of ${\mathcal E}_{b'}$ by $x'-x$ keeps ${\mathcal E}'$ in $K^o$;
hence  replacing $y$ by $y'$ sufficiently close to $y$ yields
${\mathcal E}''$
$$\theta \to a\cos \theta+b'y'\sin \theta+(x'-a)$$
in $K$.

If $x-2a\not \in K^o$, we first modify ${\mathcal E}_{b'}$ to
${\mathcal E}_{b',a'}$:
$$r_{b',a'}(\theta)= a'\cos \theta+b'y\sin \theta+(x-a')$$
with $a'-a=\delta(a-x)$ with $\delta>0$ sufficiently small so
that ${\mathcal E}_{b',a'}\subset K^o$. This is possible since the
vectors
$$
r_{b',a'}(\theta)-r_{b'}(\theta)=(a-a')(1-\cos \theta)
=\delta(1-\cos\theta) (x-a)
$$
point in the same direction for all $\theta$. Note that
$$r_{b',a'}(0)=x \ \hbox{and} \ r_{b',a'}'(0)=b'y$$
so that once again any sufficiently small translation ${\mathcal
E}'$
$$\theta \to a'\cos \theta+b'y\sin \theta+(x'-a')$$
of ${\mathcal E}_{b',a'}$ by $x'-x$ keeps ${\mathcal E}'$ in
$K^o$. Again, replacing $y$ by $y'$ sufficiently close to $y$
yields ${\mathcal E}''$
$$\theta \to a'\cos \theta+b'y'\sin \theta+(x'-a')$$
in $K$. This completes the proof that $(x,y)\to b^*(x,y)$ is
continuous.

To show that $y\to 1/b^*(x,y)$ is a norm, observe that
$b^*(x,\lambda y)={1\over \lambda}b^*(x,y)$ for $\lambda >0$ and
$0<b^*(x,y) < +\infty$ if $y\not =0\in \RR^n$. Thus
$b^*(x,0)=+\infty$ so that $1/b^*(x,y)\geq 0$ with equality
precisely when $y=0$; and $1/b^*(x,\lambda y)=\lambda /b^*(x,y)$
for $\lambda \geq 0$. To verify subaddititivity in $y$, fix $x\in
K^o$ and $y_1,y_2\in \RR^n$. Let ${\mathcal E}_1$ and ${\mathcal
E}_2$ be $b-$maximal ellipses through $x$ in the directions $y_1$
and $y_2$ as in (\ref{theta}):
$$
\theta \to r_j(\theta):=a_j\cos \theta + b_jy_j\sin \theta
+(x-a_j), \ j=1,2;
$$
here, $b_j:=b^*(x,y_j)$. Consider the convex combination
$${b_2\over b_1+b_2}r_1(\theta)+{b_1\over b_1+b_2}r_2(\theta)$$
$$={a_1b_2\over b_1+b_2}\cos \theta + {b_1b_2\over b_1+b_2}y_1\sin \theta +{b_2\over b_1+b_2}(x-a_1)$$
$$+{a_2b_1\over b_1+b_2}\cos \theta + {b_1b_2\over b_1+b_2}y_2\sin \theta +{b_1\over b_1+b_2}(x-a_2)$$
$$={a_1b_2+a_2b_1\over b_1+b_2}\cos \theta + {b_1b_2\over b_1+b_2}(y_1+y_2)\sin \theta +x-{a_1b_2+a_2b_1\over b_1+b_2}.$$
By convexity, this ellipse, through $x$ in the direction
$y_1+y_2$, lies in $K$ so that
$$b^*(x,y_1+y_2) \geq {b_1b_2\over b_1+b_2}.$$
Unwinding, this says that
$${1\over b^*(x,y_1+y_2)} \leq {1\over b_1} + {1\over b_2},$$
as desired.
\end{proof}

For future use, we mention that in $\RR^2$, if ${\mathcal E}$ is
an $a-$maximal ellipse for $K$, then either

\begin{enumerate}
\item ${\mathcal E}\cap \partial K$
contains exactly two points $a_1,a_2$, in which case the tangent
lines to ${\mathcal E}$ at $a_1,a_2$ are parallel and determine a
strip $S$ containing $K$ and ${\mathcal E}$ is an $a-$maximal
ellipse for any  rectangular truncation $T$ of $S$ with ${\mathcal
E}\subset K\subset T$; or
\item ${\mathcal E}\cap \partial K$ contains $m\geq 3$ points $a_1,...,a_m$, in which case either a subset of three
points from $\{a_1,...,a_m\}$ can be found so that the tangent
lines to ${\mathcal E}$ at these three points bound a triangle $T$
containing $K$ and ${\mathcal E}$ is an $a-$maximal ellipse for
$T$, or a rectangular truncation $T$ of a strip $S$ with
${\mathcal E}\subset K\subset T$ can be found so that ${\mathcal
E}$ is an $a-$maximal ellipse for $T$.
\end{enumerate}

\section{Main result.}
    \label{sec:main}

\vskip 3mm

For any compact set $K\subset \RR^n$ with non-empty interior, take $x\in K^o$ and $y\in \RR^n\setminus \{0\}$. Then we always have the pointwise
inequality
\begin{equation}
\label{dbdm} \delta_M(x,y)\leq \delta_B^{(i)}(x,y):=\liminf_{t\to
0^+} {V_K(x+ity)\over t}.
\end{equation}
This follows from Proposition 2.1 in \cite{blw}. In particular,
this inequality holds for any convex body $K$, with equality in
case $K$ is symmetric (as we saw in the previous section). In this section, we prove that the limit  $\lim_{t\to
0^+} {V_K(x+ity)\over t}$ exists and equals $1/b^*(x,y)$. This verifies   ``Hypothesis A'' in
\cite{rev2} for convex bodies $K\subset \RR^n$, i.e., the
inscribed ellipse method and the pluripotential-theoretic method
for obtaining Bernstein-Markov-type estimates for convex bodies
are equivalent.

Let $K$ be an arbitrary convex body in $\RR^n$. Fix $x\in K^o$ and
$y\in \RR^n\setminus \{0\}$. Take a $b-$maximal ellipse ${\mathcal
E}$ through $x$ with tangent direction $y$ at $x$. We will
normalize and assume that $y$ is a unit vector; moreover, it will
be convenient to have the center at $a$ instead of $x-a$. Thus we
write
\begin{equation}
\label{thetaparam} \theta \to r(\theta)=(x-a)\cos \theta +
b^*(x,y)y\sin \theta +a, \ \theta \in [0,2\pi]
\end{equation}
This is an $a-$maximal ellipse ${\mathcal E}$ by Proposition
\ref{maxarea}; i.e., ${\mathcal E}$ forms the real points of a
leaf $L$
\begin{equation}
\label{leaf} f(\zeta)= (x-a) [{1\over 2}(\zeta
+1/\zeta)]+b^*(x,y)y[{i\over 2}(\zeta - 1/\zeta)]+a, \ |\zeta|\geq
1
\end{equation}
of our foliation for the extremal function $V_K$. We can compare
this ``$b-$maximal'' form of the leaf with its $a-$maximal form
(\ref{eq:ansatz}):
\begin{equation}
\label{leafa} f(\zeta)= A +c\zeta +\bar c/\zeta, \ |\zeta|\geq 1,
\end{equation}
where, for simplicity, we write $c:=\rho C$ in (\ref{eq:ansatz}).
Thus, from (\ref{harmonic}), $V_K(f(\zeta)) =\log |\zeta|$ for
$|\zeta|\geq 1$.

 In these coordinates $V_K(f(\zeta))=\log {|\zeta|}$. We first show that
$$\lim_{r\to 1^+}{f(r)-f(1)\over r-1} = ib^*(x,y)y.$$
This follows from the calculation
$$f(r)-f(1)= (x-a)\bigl(\frac{(r-1)^2}{2r}\bigr)+ib^*(x,y)\frac{(r-1)(r+1)}{2r}.$$

Thus the real tangent vector to the real curve $r\to f(r), \ r\geq
1$ as $r\to 1^+$ is in the direction $ib^*(x,y)y$. Now $f(1)=x$
and $x\in K$ so $V_K(f(1))=V_K(x)=0$; and, since $f$ is a leaf of
our foliation, $V_K(f(r))=\log r$. Hence
$${V_K(f(r))-V_K(f(1))\over r-1}={\log r \over r-1}$$ so that
$$ \lim_{r\to 1^+}{V_K(f(r))-V_K(f(1))\over r-1}=\lim_{r\to 1^+}{\log r \over r-1}$$
exists and equals $1$. This elementary calculation shows that for
any convex body $K\subset \RR^n$,
\begin{equation}
\label{curvelim} \lim_{r\to 1^+}{V_K(f(r))-V_K(f(1))\over
b^*(x,y)(r-1)}= {1\over b^*(x,y)};
\end{equation}
i.e., the curvilinear limit along the curve $f(r)$ in the
direction of $iy$ at $x$ exists and equals ${1\over b^*(x,y)}$.

Note that
$$f(r)-f(1)=f(r)-x=ib^*(x,y)y(r-1)+0((r-1)^2),$$
so that the point $x+ib^{*}(x,y)y(r-1)$ is $O((r-1)^2)$ close to
the point $f(r)$. We use the explicit form (\ref{leaf}) of the
leaf to verify the existence of the limit in the directional
derivative $\delta_B(x,y)$.

\begin{theorem}
\label{main1} Let $K$ be a convex body in $\RR^n$. Then the limit
in the definition of the directional derivative exists and equals
${1\over b^*(x,y)}$:
$$\delta_B(x,y):=\lim_{t\to 0^+}{V_K(x+ity)\over t}={1\over b^*(x,y)}.$$
\end{theorem}

\begin{proof} If we can show
\begin{equation}
\label{straightlim} \lim_{r\to
1^+}{V_K(f(r))-V_K(x+ib^*(x,y)y(r-1))\over b^*(x,y)(r-1)}=0,
\end{equation}
then using (\ref{curvelim}) and the preceeding discussion, we will
have
\begin{equation}
\label{dirlim} \lim_{t\to 0^+}{V_K(x+ity)\over t}={1\over
b^*(x,y)}.
\end{equation}

We first consider the case when $K$ admits a continuous foliation;
i.e., $K\in {\mathcal C}$. Consider a fixed point
$w:=x+ib^{*}(x,y)y(r-1)\in\CC^n$. This belongs to some foliation
leaf $M$ which we write in the form (\ref{leafa}):
$$
g(\zeta)=\alpha +\gamma \zeta + \bar {\gamma} /\zeta: ~~\CC
\setminus \Delta \to M \subset \CC^n.
$$
We need to use the facts that when $r\to1^+$, then $w\to x\in
L$, and, by continuity of the foliation, the leaf parameters
for $(g,M)$ should converge to those of $(f,L)$; i.e., $\alpha \to
A$ and $\gamma \to c$. We remark that if we compare (\ref{leaf})
and (\ref{leafa}), writing $b:=b^*(x,y)$ we have the relations
\begin{equation}
\label{relations} A=a \ \hbox{and} \ c = \frac{1}{2} (x-a +iby).
\end{equation}
Here we supress a rotational invariance: the substitution
$\zeta':=\zeta e^{i\varphi}$ for any fixed constant $\varphi$
describes the same leaf with a different parametrization; thus we
fix its value so that
$$
\xi:=g(1)= 2\Re \, \gamma +\alpha
$$
is closest to $x:=f(1)=2\Re \, c + A$, i.e., $|g(1)-f(1)|\leq
|g(e^{i\theta})-f(1)|$ for all $\theta$. To emphasize, writing the
leaf $(f,L)$ in $b-$maximal form (\ref{leaf}),
$$
f(\zeta)=(x-a)\frac12 (\zeta+\frac1\zeta) + b y \frac{i}2
(\zeta-\frac1\zeta) + a
$$
where, from (\ref{relations}) and the fact that $y$ is a unit
vector, $b:= 2|\Im \, c|>0$ and $y:= \frac 2b \Im \, c \in \RR^n$.
Now, apriori, we do not know if $(g,M)$ is $b-$maximal
(aposteriori, it is: see Corollary \ref{amaxisbmax}). However, we
may still write this leaf in the form
$$
g(\zeta)=(\xi-\alpha)\frac12 (\zeta+\frac1\zeta) + \beta \eta
\frac{i}2 (\zeta-\frac1\zeta) + \alpha
$$
with $\beta:= 2|\Im \, \gamma|>0$ and $\eta:= \frac 2{\beta} \Im
\, \gamma \in \RR^n$. Note that continuity of the foliation
implies $\beta >0$ since $b>0$; indeed, $\beta \sim b$, $\xi \sim
x$, $\eta \sim y$, $\alpha\sim a$, and $\gamma\sim c$.

Since $w\in M$, there is a point $\omega\in \CC \setminus \Delta$
with $g(\omega)=w$. Our task is to calculate
$V_K(w)=V_K(g(\omega))$. On a leaf of the foliation we have the
formula $V_K(g(\omega))=\log|\omega|$, so it suffices to compute
$\log|\omega|$. The representation of $w$ as $g(\omega)$ means
that for $j=1,\dots,n$,
$$
x_j+ib y_j(r-1)=w_j=g_j(\omega)=(\xi_j-\alpha_j)\frac12
(\omega+\frac1\omega) + \beta \eta_j \frac{i}2
(\omega-\frac1\omega) + \alpha_j.
$$
Since $y$ and $\eta$ are unit vectors which are close to each
other, we can choose a coordinate $j$ with $y_j\ne 0$, $\eta_j\ne
0$. For this coordinate $j$, the previous displayed equation gives
$$
\frac12 (\xi_j-\alpha_j+i\beta \eta_j) \omega^2 +
(\alpha_j-x_j-iby_j(r-1))\omega + \frac12 (\xi_j-\alpha_j-i\beta
\eta_j) =0,
$$
a quadratic equation in $\omega$. Corresponding to the double
mapping properties of the Joukowski map $\frac12(\zeta+1/\zeta)$,
there are two roots, one in $|\zeta|<1$ and one in $|\zeta|>1$,
the latter being our $\omega$ as we considered the mapping of the
exterior of the unit disc $\Delta$. For convenience, put
$\rho:=b(r-1)y_j$. Since $by_j\not =0$, $\rho \asymp r-1$. By the
quadratic formula,
\begin{align*}
\omega_{1,2}& = \frac{x_j-\alpha_j+i\rho\pm
\sqrt{(\alpha_j-x_j-i\rho)^2-(\xi_j-\alpha_j)^2-\beta^2
\eta_j^2}}{\xi_j-\alpha_j+i\beta \eta_j}.
\end{align*}

Set $Q:=\beta^2\eta_j^2 +(\xi_j-\alpha_j)^2 - (x_j-\alpha_j)^2\sim
b^2y_j^2>0$ by continuity of the leaf parameters and choice of
$j$. Using this and the simple formula
$\sqrt{A+2B}=\sqrt{A} + B/\sqrt{A} + O(B^2/A^{3/2})$, valid
uniformly for $|B|<A/3$, say, we can rewrite the square root as
\begin{align*}
& \pm \sqrt{(x_j-\alpha_j)^2-i2(\alpha_j-x_j)\rho - \rho^2-(\xi_j-\alpha_j)^2-\beta^2 \eta_j^2} \\
& \qquad = \pm i \sqrt{Q+2i(\alpha_j-x_j)\rho+\rho^2} \\
& \qquad = \pm i \left\{\sqrt{Q}
+i\frac{(\alpha_j-x_j)\rho}{\sqrt{Q}}
+ O(\rho^2)\right\} \\
& \qquad = \pm \left\{\frac{(x_j-\alpha_j)\rho}{\sqrt{Q}} + i
\sqrt{Q} + O(\rho^2)\right\} .
\end{align*}
Put $P:=\xi_j-\alpha_j+i\beta \eta_j$.  Then
\begin{align*}
|\omega_{1,2}P|^2 & = \bigl|(x_j-\alpha_j)(1\pm
\frac{\rho}{\sqrt{Q}}) + i\bigl[ \pm
\sqrt{Q}+ \rho \bigr]+O(\rho^2)\bigr|^2 \\
& = \bigl[(\alpha_j-x_j)\pm\frac{\rho(\alpha_j-x_j)}{\sqrt{Q}}\bigr]^2 + \bigl[\pm\sqrt{Q}+\rho\bigr]^2 + O(\rho^2)\\
&  = (\alpha_j-x_j)^2\pm \frac{2\rho(\alpha_j-x_j)^2}{\sqrt{Q}} + Q  \pm 2\rho\sqrt{Q} +O(\rho^2) \\
&  = (\alpha_j-x_j)^2 + Q \pm \frac{2\rho}{\sqrt{Q}}\bigl( Q + (\alpha_j-x_j)^2\bigr) + O(\rho^2). \\
\end{align*}
We have the identity $|P|^2=(\alpha_j-x_j)^2+Q$; dividing by this
quantity on both sides yields
$$
|\omega_{1,2}|^2 = 1 \pm \frac{2\rho}{\sqrt{Q}} + O(\rho^2).
$$
Fixing the branch of the square-root with $\sqrt{Q}>0$, it is
clear that among the choice of signs in $\pm$ the one equal to the
sign of $y_j$ leads to the larger absolute value (and the one with
$|\omega|$ exceeding 1); hence for such $\omega$ with $g(\omega)=w$,
\begin{align*}
|\omega|^2 & = 1 + \frac{2 |\rho|}{\sqrt{Q}}+ O(\rho^2) = 1 +
\frac{2b|y_j|(r-1)}{\sqrt{Q}}+ O((r-1)^2),
\end{align*}
so that
$$
\log|\omega|^2 = \log\bigl|1 + \frac{2b|y_j|(r-1)}{\sqrt{Q}}+
O((r-1)^2)\bigr| = \frac{2b|y_j|(r-1)}{\sqrt{Q}}+ O((r-1)^2).
$$
Hence
$$
\frac{V_K(x+iby(r-1))}{r-1} = \frac{V_K(g(\omega))}{r-1}
=\frac{\log|\omega|}{r-1} = \frac{1}{2}\frac{\log|\omega|^2}{r-1}
= \frac{b|y_j|}{\sqrt{Q}} + O(r-1).
$$
Recall once again that the continuity of the foliation, as
$r\to 1$ we have $\xi_j\to x_j$, $\alpha_j\to a_j$, $\beta\to b$,
and thus $\sqrt{Q}\to b|y_j|$.  Hence
$$\lim_{r\to 1^+}\frac{V_K(x+iby(r-1))}{r-1} = 1 = \lim_{r\to 1^+}\frac{V_K(f(r))}{r-1}.$$
This verifies (\ref{straightlim}) and hence (\ref{dirlim}) in the
case when $K$ admits a continuous foliation.

For a general convex body $K$, we use the previous case and an
appropriate approximation argument to verify (\ref{dirlim}). To
emphasize the set(s) under discussion, we write
$$b^*(K;x,y):=b^*(x,y) \ \hbox{for the set} \ K.$$
We need to prove two inequalities:
\begin{equation}\label{liminf}
\frac{1}{b^{*}(K;x,y)} \leq \liminf_{t\to 0+} \frac{V_K(x+ity)}{t}
\end{equation}
and
\begin{equation}\label{limsup}
\frac{1}{b^{*}(K;x,y)} \geq \limsup_{t\to 0+}
\frac{V_K(x+ity)}{t}.
\end{equation}

Note that if $K$ and $\kappa $ are two convex bodies, we have for
any $y\in \RR^n$ the inequalities
\begin{equation}\label{bmonotonicity}
b^{*}(K;x,y) \begin{cases} & \leq b^{*}(\kappa;x,y) \qquad
\textrm{if}\qquad   K\subset \kappa \\ & \geq b^{*}(\kappa;x,y)
\qquad \textrm{if}\qquad  \kappa \subset K
\end{cases},
\end{equation}
and for any $y\in \RR^n$ and any $t\in \RR$ the inequalities
\begin{equation}\label{Vmonotonicity}
V_K(x+ity) \begin{cases} & \geq V_\kappa(x+ity) \qquad
\textrm{if}\qquad  K \subset \kappa \\
& \leq V_\kappa(x+ity) \qquad \textrm{if}\qquad  \kappa \subset K
\end{cases}.
\end{equation}

Fix $\alpha<1$ arbitrarily close to $1$ and choose $\delta>0$
small (to be determined later in terms of $\alpha$). From the
discussion at the end of section 2, $\mathcal C$ is dense in
$\mathcal K$; thus we can find $\kappa \in \mathcal C$ such that
the Hausdorff distance between $\kappa$ and  $K$ is at most
$\delta$. Take an $\alpha$-dilated (at $x$) copy $K_1$ of $K$
and an $\alpha$-dilated copy $\kappa_1$ of $\kappa$. Then $x\in
K_1^{0}$, and if $\delta=\delta(\alpha)$ is sufficiently small, we
have $\kappa_1 \subset K$. We also take the $1/\alpha$-dilated
copies $K_2$ and $\kappa_2$ of $K$ and $\kappa$. Again for small
enough $\delta$, we will have $K\subset \kappa_2$. Note that
$\kappa_2$ is the $\alpha^{-2}$-dilated copy of $\kappa_1$; hence
$b^{*}(\kappa_2;x,y)= \alpha^{-2} b^{*}(\kappa_1;x,y)$. Therefore,
using (\ref{dirlim}) for $\kappa_1$ and $\kappa_2$, we obtain
\begin{equation}\label{gcompare}
\lim_{t\to 0+}
\frac{V_{\kappa_2}(x+ity)}{t}=\frac{1}{b^{*}(\kappa_2;x,y)}=\alpha^2
\frac{1}{b^{*}(\kappa_1;x,y)}={\alpha^2} \lim_{t\to 0+}
\frac{V_{\kappa_1}(x+ity)}{t}.
\end{equation}
Using \eqref{gcompare}, \eqref{bmonotonicity} and
\eqref{Vmonotonicity}, we obtain
\begin{align}\label{aliminf}
\frac{1}{b^{*}(K;x,y)}& \leq \frac{1}{b^{*}(\kappa_1;x,y)}=
\lim_{t\to 0+} \frac{V_{\kappa_1}(x+ity)}{t}  \notag \\ & =
\frac{1}{\alpha^2 } \lim_{t\to 0+} \frac{V_{\kappa_2}(x+ity)}{t} =
\frac{1}{\alpha^2 }
\liminf_{t\to 0+} \frac{V_{\kappa_2}(x+ity)}{t} \notag \\
&\leq \frac{1}{\alpha^2 } \liminf_{t\to 0+}
\frac{V_{K}(x+ity)}{t};
\end{align}
and, in a similar fashion we get
\begin{align}\label{alimsup}
\frac{1}{b^{*}(K;x,y)}& \geq \frac{1}{b^{*}(\kappa_2;x,y)}=
\lim_{t\to 0+} \frac{V_{\kappa_2}(x+ity)}{t}  \notag \\ & =
{\alpha^2 } \lim_{t\to 0+} \frac{V_{\kappa_1}(x+ity)}{t} =
{\alpha^2 }
\limsup_{t\to 0+} \frac{V_{\kappa_1}(x+ity)}{t} \notag \\
& \geq {\alpha^2 } \limsup_{t\to 0+} \frac{V_{K}(x+ity)}{t}.
\end{align}
Since $\alpha$ can be taken arbitrarily close to $1$,
\eqref{liminf} and \eqref{limsup} follow from \eqref{aliminf} and
\eqref{alimsup}.

\end{proof}

\begin{remark} Observe that the essential property used to verify
(\ref{straightlim}) for $K\in {\mathcal C}$ is (\ref{harmonic});
i.e., that $V_K(f(\zeta))=\log |\zeta|$ on $L$ (i.e., for
$|\zeta|\geq 1$), which is equivalent to the $a-$maximality of the
real ellipse ${\mathcal E}\subset L$ (see section 2).
\end{remark}

\begin{corollary}
\label{amaxisbmax} For any convex body $K$, an ellipse ${\mathcal
E}\subset K$ is $a-$maximal if and only if it is $b-$maximal for
all $x\in K^o$ and $y\in T_x{\mathcal E}$.
\end{corollary}
\vskip4pt
\begin{proof}
That $b-$maximality implies $a-$maximality was proved in
Proposition \ref{maxarea}. For the converse, we first suppose that
$K\in {\mathcal C}$. Let ${\mathcal E}$ be an $a-$maximal ellipse.
Fix $x\in K^o$ and $y\in T_x{\mathcal E}$ a unit vector. Then
${\mathcal E}=f(\partial \Delta)$ where
\begin{equation}
\label{bform} f(\zeta)=(x-\alpha)\frac12 (\zeta+\frac1\zeta) +
\beta y \frac{i}2 (\zeta-\frac1\zeta) + \alpha
\end{equation}
and $V_K(f(\zeta))=\log |\zeta|$ for $|\zeta|\geq 1$, so that,
from the remark,
$$1/\beta = \lim_{r\to 1^+}{V_K(f(r))\over \beta(r-1)}=\lim_{r\to 1^+}{V_K(x+i\beta y(r-1))\over \beta (r-1)}.$$
Writing $t=\beta(r-1)$ in the limit on the right and using Theorem
\ref{main1},
$$1/\beta =\lim_{t\to 0^+}{V_K(x+ity)\over t}={1\over b^*(x,y)}$$
so that ${\mathcal E}$ is $b-$maximal for $x,y$.

Now let $K\in {\mathcal K}$ be an arbitrary convex body. We
consider first the case where ${\mathcal E}$ is the unique
$a-$maximal ellipse for $[c]\in H$; i.e., for its orientation and
eccentricity, and we again write ${\mathcal E}=f(\partial \Delta)$
as in (\ref{bform}). Take a sequence $\{K_j\}\subset {\mathcal C}$
with $K_j \searrow K$. For each $j$, let ${\mathcal E}_j$ be the
unique $a-$maximal ellipse for $K_j, [c]$, and let $f_j$ denote
the corresponding leaf. Then (cf., \cite{blm2}) $f_j\to f$
uniformly so that ${\mathcal E}_j\to {\mathcal E}$. As in the
proof of Theorem \ref{main1}, we may write
$$ f_j(\zeta)=(x_j-\alpha_j)\frac12 (\zeta+\frac1\zeta) + \beta_j y_j \frac{i}2
(\zeta-\frac1\zeta) + \alpha_j.
$$
By the first part of the proof, ${\mathcal E}_j$ is $b$-maximal in
$K_j$, so that $\beta_j=b^*(x_j,y_j,K_j)$. The uniform convergence
of $f_j$ to $f$ implies that $\alpha_j \to \alpha$, $x_j\to x$ and
$y_j\to y$. Moreover, since $x\in K^o$, for $j$ sufficiently
large, $x_j \in K^o$. From the  continuity of $b^*$ (Proposition
\ref{dbnorm}) and the fact that $K\subset K_j$,
$$b^*(x,y,K)=\lim_{j\to \infty}b^*(x_j,y_j,K)\leq \lim_{j\to \infty}b^*(x_j,y_j,K_j) =\lim_{j\to \infty} \beta_j =\beta.$$
Hence $\beta = b^*(x,y,K)$ and ${\mathcal E}$ is $b-$maximal.

In the case where ${\mathcal E}$ is not the unique $a-$maximal
ellipse for the corresponding $[c]\in H$, it is an $a-$maximal
ellipse for this $[c]$ and a ``strip'' $S$, i.e., a closed body
$S$ bounded by two parallel hyperplanes $P_1,P_2$ with $K\subset
S$ (see section 7 of \cite{blm2}).  If $\mathcal E$ is given by
$$
r(\theta) = f(e^{i\theta}) = (x-\alpha)\cos\theta + \beta y
\sin\theta + \alpha
$$
then there is $\theta_0$ such that $r(\theta_0)\in P_1$ and
$r(\theta_0+\pi)\in P_2$. It is therefore sufficient to show that
any ellipse ${\mathcal E}\subset K$ that intersects $P_1,P_2$ is
$b$-maximal for any $x\in {\mathcal E}\cap K^o$, $y\in
T_x{\mathcal E}$.  Take a sufficiently large convex set $T$ that
is symmetric about the center of the ellipse (e.g., a large box)
so that $K\subset T\subset S$.  Clearly $\mathcal E$ is
$a$-maximal for $T$. Since $T$ is symmetric, $T\in{\mathcal C}$,
so by the first part of the proof, $\mathcal E$ is $b$-maximal for
$T,x,y$.  Hence it is also $b$-maximal for $K,x,y$.
\end{proof}

We turn to the Monge-Ampere measure. We know that $(dd^cV_K)^n$ is
supported in $K$ and is absolutely continuous with respect to
Lebesgue measure $dx$ on $K^o$, i.e., $(dd^cV_K)^n=\tilde
\lambda(x)dx$ for a locally integrable non-negative function
$\tilde c$ on $K^o$ \cite{bt}. Baran has proved the following (see
\cite{bar2}, Propositions 1.10, 1.11 and Lemma 1.12).

\begin{proposition}
\label{propbaran} Let $D\subset {\bf C}^n$ and let $\Omega :=D\cap
\RR^n$.  Let $u$ be a nonnegative psh function on $D$ which
satisfies:
\end{proposition}
$\begin{array}{rl}
 i. & \Omega=\{u=0\} \\
ii.  & (dd^cu)^n=0 \; {\mbox{on}} \; D\setminus \Omega \\
iii. & (dd^cu)^n=\lambda(x)dx \; on \; \Omega \; where \; c\in L^1_{loc}(\Omega) \\
iv. & for \; all \; x\in \Omega, \ y \in \RR^n, \; the \; limit
\end{array}$
$$h(x,y):=\lim_{t\to 0^+} {u(x+ity)\over t}  \; exists \; and \; is \; continuous \; on \; \Omega \times i\RR^n$$

$\begin{array}{rl} v. & x\in \Omega, y\to h(x,y) \; is \; a \;
norm.
\end{array}$

\noindent {\em Then} 
$$\lambda(x)=n! {\rm vol} \{y:h(x,y)\leq 1\}^*$$
{\em and $\lambda(x)$ is a continuous function on $\Omega$.}

\vskip4pt We now obtain the generalization of (\ref{eq:monge}).

\begin{corollary}
\label{cor:mt} Let $K$ be a convex body in $\RR^n$. Then
$$(dd^cV_K)^n=\lambda(x)dx \ \hbox{for} \ x\in K^o$$
where $\lambda(x)=n! {\rm vol}(\{y: \delta_B(x,y)={1\over
b^*(x,y)} \leq 1\}^*)$ is continuous. Moreover, $(dd^cV_K)^n$ puts
no mass on the boundary $\partial K$ (relative to $\RR^n$).
\end{corollary}

\begin{proof}
The formula for $(dd^cV_K)^n$ on $K^o$ is immediate from Theorem
\ref{main1}, Proposition \ref{propbaran} and the paragraph
preceeding it, and Proposition \ref{dbnorm}. To show that
$(dd^cV_K)^n$ puts no mass on the boundary $\partial K$, we
proceed as in \cite{bt}. Let $\{K_j\}$ be a sequence of convex
bodies in $\RR^n$ with real-analytic boundaries $\{\partial K_j\}$
such that $K_j$ increases to $K$. Then $\partial K_j$ is
pluripolar so that $(dd^cV_{K_j})^n$ puts no mass on $\partial
K_j$ (cf., Proposition 4.6.4 of \cite{kli}). Writing
$(dd^cV_{K_j})^n:=\lambda_j(x)dx$ for $x\in K_j^o$, we have
$$(2\pi)^n= \int_{K_j^o} \lambda_j(x)dx =\int_{K^o}\lambda_j(x)dx$$
where we extend $\lambda_j(x)$ to be zero outside of $K_j^o$. Since
$\int_{K^o}\lambda_j(x)dx \leq (2\pi)^n <\infty$, and the density
functions $\lambda_j(x)$ increase almost everywhere on $K^o$ to
$\lambda(x)$ (cf., (\ref{bmonotonicity})), by dominated
convergence we have
$$(2\pi)^n=\lim_{j\to \infty}\int_{K^o}\lambda_j(x)dx = \int_{K^o}\lambda(x)dx .$$
Thus $(dd^cV_K)^n$ puts no mass on the boundary $\partial K$.
\end{proof}

\vskip4pt

We end this note with a final remark on Bernstein-Markov  inequalities.
Baran \cite{bar2} conjectured that we
have equality $\delta_M(x,y)= \delta_B^{(i)}(x,y)$ in (\ref{dbdm})
for general convex bodies. With respect to this conjecture, we
make the following observation: {\sl if we can prove
$\delta_M(x,y)= \delta_B(x,y)$ for a triangle $T$ in $\RR^2$, then
equality holds for all convex bodies in $\RR^2$}. For let $K$ be a
convex body in $\RR^2$. Fix $x\in K^o$ and $|y|=1$. Take a
$b-$maximal ellipse ${\mathcal E}={\mathcal E}(x,y)$ for $K$ with
parameter $b=b^*(x,y)$ and, as in (1) or (2) at the end of section
3, take a rectangle or triangle $T$ containing $K$ in which
${\mathcal E}$ is an $a-$maximal ellipse. Then since $
\delta_B^{T}(x,y) ={1\over b^*(x,y)} =\delta_B^{K}(x,y)$, we have
$$\delta_M^{K}(x,y)\geq \delta_M^{T}(x,y)= \delta_B^{T}(x,y) =\delta_B^{K}(x,y).$$
From (\ref{dbdm}), $\delta_M^{K}(x,y)\leq \delta_B^{K}(x,y)$ and
equality holds.

\end{document}